\documentclass[12pt]{article}

\usepackage{amsmath, epsfig, cite, setspace}
\usepackage{amssymb}
\usepackage{amsfonts}
\usepackage{latexsym}
\usepackage{amsthm}

\makeatletter
\renewcommand{\@seccntformat}[1]{{\csname the#1\endcsname}.\hspace{.5em}}
\makeatother

\newtheorem{theorem}{Theorem}[section]

\newtheorem{corollary}[theorem]{Corollary}

\newtheorem{lemma}[theorem]{Lemma}

\renewcommand{\thefootnote}{*}

\numberwithin{equation}{section}

\begin{document}

\begin{center}
{\large\bf  $q$-Supercongruences for multidimensional series modulo
the sixth power of a cyclotomic polynomial }
\end{center}

\vskip 2mm \centerline{Chuanan Wei }
\begin{center}
{\footnotesize School of Biomedical Information and Engineering,\\
  Hainan Medical University, Haikou 571199, China\\

{\tt Email address:weichuanan78@163.com } }
\end{center}


\vskip 0.7cm \noindent{\bf Abstract.} With the help of El
Bachraoui's lemma, the creative microscoping method, and a new form
of the Chinese remainder theorem for coprime polynomials, we prove a
$q$-supercongruence for double series and a $q$-supercongruence for
triple series modulo the sixth power of a cyclotomic polynomial. As
conclusions, two corresponding supercongruences for double and
triple series, which are associated with the (D.2) supercongruence
of Van Hamme, are given.



\renewcommand{\thefootnote}{**}

\section{Introduction}
For a complex number $x$ and a nonnegative integer $n$, define the
shifted-factorial as
\[(x)_{n}=\Gamma(x+n)/\Gamma(x),\]
where $\Gamma(x)$ is the usual Gamma function.  For convenience, let
$p$ stand for any odd prime throughout the paper. In 1997, Van Hamme
\cite[(D. 2)]{Hamme} displayed the  interesting conjecture: for
$p\equiv 1\pmod 6$,
\begin{equation}\label{vanhamme}
\sum_{k=0}^{(p-1)/3}(6k+1)\frac{(\frac{1}{3})_k^6}{(1)_k^6}\equiv
 -p\,\Gamma_p(\tfrac{1}{3})^9 \pmod{p^4}.
\end{equation}
Nineteen years later, Long and Ramakrishna \cite[Theorem 2]{LR}
verified the following generalization of \eqref{vanhamme}:
\begin{equation}\label{eq:long}
\sum_{k=0}^{p-1}(6k+1)\frac{(\frac{1}{3})_k^6}{(1)_k^6}\equiv
\begin{cases}\displaystyle -p\,\Gamma_p(\tfrac{1}{3})^9  \pmod{p^6}, &\text{if $p\equiv
1\pmod 6$,}
\\[7pt]
\displaystyle -\frac{10}{27}p^4\,\Gamma_p(\tfrac{1}{3})^9\pmod{p^6},
&\text{if $p\equiv 5\pmod 6$.}
\end{cases}
\end{equation}
 Some results
and conjectures related to \eqref{eq:long} can be seen in Guo, Liu,
and Schlosser \cite{GLS}.

For two complex numbers $x$ and $q$ with $|q|<1$ and a nonnegative
integer $n$, define the $q$-shifted factorial
 to be
  \begin{equation*}
(x;q)_{\infty}=\prod_{k=1}^{\infty}(1-xq^k)\quad\text{and}\quad
(x;q)_n=\frac{(x;q)_{\infty}}{(xq^n;q)_{\infty}}.
 \end{equation*}
For simplicity, we sometimes use the compact notation
\begin{equation*}
(x_1,x_2,\dots,x_m;q)_{t}=(x_1;q)_{t}(x_2;q)_{t}\cdots(x_m;q)_{t},
 \end{equation*}
where $m\in\mathbb{Z}^{+}$ and $t\in\mathbb{Z}^{+}\cup\{0,\infty\}.$
In 2021, Guo and Schlosser \cite[Theorem 2.3]{GS} found a partial
$q$-analogue of \eqref{eq:long}:  for any positive integer $n$,
\begin{equation*}
\sum_{k=0}^{n-1}[6k+1]\frac{(q;q^3)_k^6}{(q^3;q^3)_k^6}q^{3k}\equiv
\begin{cases} \displaystyle 0  \pmod{[n]}, &\text{if $n\equiv 1\pmod 3$,}\\[10pt]
\displaystyle  0 \pmod{[n]\Phi_n(q)}, &\text{if $n\equiv 2\pmod 3$,}
\end{cases}
\end{equation*}
where $[s]$ is the $q$-integer $(1-q^s)/(1-q)$ and $\Phi_n(q)$
denotes the $n$-th cyclotomic polynomial in $q$. In terms of the set
of $q$-congruence relations:
\begin{align}
&\frac{(1-bq^{tn})(b-q^{tn})(-1-a^2+aq^{tn})}{(a-b)(1-ab)}\equiv1\pmod{(1-aq^{tn})(a-q^{tn})},
\label{relations-a}\\[5pt]\label{relations-b}
&\frac{(1-aq^{tn})(a-q^{tn})(-1-b^2+bq^{tn})}{(b-a)(1-ba)}\equiv1\pmod{(1-bq^{tn})(b-q^{tn})},
\end{align}
where $t\in\{1,2\}$, and some other tools, the author \cite[Theorems
1.1 and 1.2]{Wei-a} catched hold of the following stronger
conclusions: for any positive integer $n\equiv 1\pmod 3$,
\begin{align}
\sum_{k=0}^{(n-1)/3}[6k+1]\frac{(q;q^3)_k^6}{(q^3;q^3)_k^6}q^{3k}
&\equiv[n] \frac{(q^2;q^3)_{(n-1)/3}^3}{(q^3;q^3)_{(n-1)/3}^3}
\notag\\&\:\hspace{-30mm}\times\:
\bigg\{1+[n]^2(2-q^{n})\sum_{r=1}^{(n-1)/3}\bigg(\frac{q^{3r-1}}{[3r-1]^2}-\frac{q^{3r}}{[3r]^2}\bigg)\bigg\}
\pmod{[n]\Phi_n(q)^4}, \label{wei-a}
\end{align}
and for any positive integer $n\equiv 2\pmod 3$,
\begin{align}
\sum_{k=0}^{M}[6k+1]\frac{(q;q^3)_k^6}{(q^3;q^3)_k^6}q^{3k} &\equiv
5[2n]
\frac{(q^2;q^3)_{(2n-1)/3}^3}{(q^3;q^3)_{(2n-1)/3}^3}\pmod{[n]\Phi_n(q)^5}.
 \label{wei-b}
\end{align}
It is clear that \eqref{wei-a} is still a partial $q$-analogue of
the first case in \eqref{eq:long}, although \eqref{wei-b} has been a
complete $q$-analogue of the second case in \eqref{eq:long}.

In 2021, El Bachraoui \cite{Mohamed} discovered  a few
$q$-supercongruences for double series including the result: for any
positive integer $n$ with $\gcd(n,6)=1$,
\begin{equation*}
\sum_{k=0}^{n-1}\sum_{j=0}^{k}A_q(j)A_q(k-j)\equiv q[n]^2
\pmod{[n]\Phi_n(q)^2},
\end{equation*}
where
\begin{align*}
A_q(k)=[8k+1]\frac{(q;q^2)_k^2
(q;q^2)_{2k}}{(q^2;q^2)_{2k}(q^6;q^6)_k^2}q^{2k^2}.
\end{align*}
Via the series rearrangement method, it is not difficult to
understand that the last formula can be written as the following
symmetric form: for $i,j\geq0$,
\begin{equation*}
\sum_{i+j\leq n-1}A_q(i)A_q(j)\equiv q[n]^2 \pmod{[n]\Phi_n(q)^2}.
\end{equation*}
 Further $q$-supercongruences for multidimensional series can be
seen in the papers \cite{Li,GuoLi,Mohamed-a,Wei-b}. For more
$q$-analogues of supercongruences, we refer the reader to the papers
\cite{Guo-a,Guo-b,Guo-c,GS-a,GuoZu-c,NW,LP,Wei-a}.

By introducing another set of $q$-congruence relations, which are
different from \eqref{relations-a} and \eqref{relations-b} and will
appear in the next section, we shall prove the following
$q$-supercongruence for double series modulo the sixth power of a
cyclotomic polynomial, which has been conjectured by the author and
Li \cite{Wei-b}.

\begin{theorem}\label{thm-a}
Let $n$ be a positive integer such that $n\equiv1\pmod{3}$ and
$$A_q(k)=[6k+1]\frac{(q;q^3)_k^6}{(q^3;q^3)_k^6}q^{3k}.$$
Then for $i,j\geq0$, modulo $\Phi_n(q)^6$,
\begin{align}
\sum_{i+j\leq n-1}A_q(i)A_q(j)&\equiv
[n]^2\frac{(q^2;q^3)_{(n-1)/3}^6}{(q^3;q^3)_{(n-1)/3}^6}
\notag\\
&\quad\times\:
\bigg\{1+2[n]^2(2-q^{n})\sum_{r=1}^{(n-1)/3}\bigg(\frac{q^{3r-1}}{[3r-1]^2}-\frac{q^{3r}}{[3r]^2}\bigg)\bigg\}.
\label{eq:thm-a}
\end{align}
\end{theorem}

Choosing $n=p$ and taking $q\to 1$ in Theorem \ref{thm-a}, we obtain
the following supercongruence associated with the first case in
\eqref{eq:long}.

\begin{corollary}\label{cor-a}
Let $p\equiv1\pmod{6}$ be a prime.  Then for $i,j\geq0$,
\begin{align*}
&\sum_{i+j\leq p-1}
(6i+1)(6j+1)\frac{(\frac{1}{3})_i^6\,(\frac{1}{3})_j^6}{(1)_i^6\,(1)_j^6}\\
&\quad\equiv\frac{(\frac{2}{3})_{(p-1)/3}^6}{(1)_{(p-1)/3}^6}
\bigg\{p^{2}+2p^{4}\sum_{r=1}^{(p-1)/3}\bigg(\frac{1}{(3r-1)^2}-\frac{1}{(3r)^2}\bigg)\bigg\}\pmod{p^{6}}.
\end{align*}
\end{corollary}

Similarly, we shall also prove the following $q$-supercongruence for
triple series modulo the sixth power of a cyclotomic polynomial,
which was essentially conjectured by the author and Li \cite{Wei-b}.

\begin{theorem}\label{thm-aa}
Let $n$ be a positive integer subject to $n\equiv1\pmod{3}$. Then
for $i,j,k\geq0$, modulo $\Phi_n(q)^6$,
\begin{align}
\sum_{i+j+k\leq n-1}A_q(i)A_q(j)A_{q}(k)&\equiv
[n]^3\frac{(q^2;q^3)_{(n-1)/3}^9}{(q^3;q^3)_{(n-1)/3}^9}
\notag\\
&\quad\times\:
\bigg\{1+3[n]^2\sum_{r=1}^{(n-1)/3}\bigg(\frac{q^{3r-1}}{[3r-1]^2}-\frac{q^{3r}}{[3r]^2}\bigg)\bigg\},
\label{eq:thm-aa}
\end{align}
where $A_{q}(k)$ has been given in Theorem \ref{thm-a}.
\end{theorem}

Define the harmonic numbers of $2$-order by
\[H_{n}^{(2)}
  =\sum_{r=1}^n\frac{1}{r^{2}}.\]
It is ordinary to show that
 \begin{align}
\sum_{r=1}^{(p-1)/3}\bigg(\frac{1}{(3r-1)^2}-\frac{1}{(3r)^2}\bigg)&=H_{p-1}^{(2)}-\sum_{r=1}^{(p-1)/3}\frac{1}{(3r-2)^2}-\frac{2}{9}H_{(p-1)/3}^{(2)}
 \notag\\[5pt]
 &\equiv
-\sum_{r=1}^{(p-1)/3}\frac{1}{(3r-2)^2}-\frac{2}{9}H_{(p-1)/3}^{(2)}
\notag\\[5pt]
&=-\sum_{r=1}^{(p-1)/3}\frac{1}{(p-3r)^2}-\frac{2}{9}H_{(p-1)/3}^{(2)}
 \notag\\[5pt]
 &\equiv
-\frac{1}{3}H_{(p-1)/3}^{(2)}
 \pmod{p}. \label{harmonic}
\end{align}

Fixing $n=p$ and taking $q\to 1$ in Theorem \ref{thm-aa}, we get the
following supercongruence related to the first case in
\eqref{eq:long} after utilizing \eqref{harmonic}.

\begin{corollary}\label{cor-aa}
Let $p\equiv1\pmod{6}$ be a prime. Then for $i,j,k\geq0$,
\begin{align*}
&\sum_{i+j+k\leq p-1}
(6i+1)(6j+1)(6k+1)\frac{(\frac{1}{3})_i^6\,(\frac{1}{3})_j^6\,(\frac{1}{3})_k^6}{(1)_i^6\,(1)_j^6\,(1)_k^6}
\\
&\quad\:\:\equiv\frac{(\frac{2}{3})_{(p-1)/3}^9}{(1)_{(p-1)/3}^9}
\bigg\{p^{3}-p^{5}H_{(p-1)/3}^{(2)}\bigg\}\pmod{p^{6}}.
\end{align*}
\end{corollary}

The structure of the paper is arranged as follows. By means of El
Bachraoui's lemma, the creative microscoping method from Guo and
Zudilin \cite{GuoZu-b},  and a new form of the Chinese remainder
theorem for coprime polynomials, we shall prove Theorem \ref{thm-a}
in Section 2. Similarly, the proof of Theorem \ref{thm-aa} will be
provided in Section 3.

\section{Proof of Theorem \ref{thm-a}}

For the aim of proving Theorem \ref{thm-a}, we demand the following
two lemmas, where Lemma \ref{lemma-a} is due to El Bachraoui
\cite{Mohamed-a} (see also \cite{WY-a}).

\begin{lemma}\label{lemma-a}
Let $m,n,t$ be positive integers with $m\geq t\geq2$ and
$n\equiv1\pmod{m}$. Let $\{\lambda(k)\}_{k=0}^\infty$ a complex
sequence. If $\lambda(k)=0$ for $(n+m-1)/m\leq k\leq n-1$, then
$$
\sum_{k_1+k_2+\cdots+k_t\leq
n-1}\lambda(k_1)\lambda(k_2)\cdots\lambda(k_t)
=\Bigg(\sum_{k=0}^{\frac{n-1}{m}}\lambda(k)\Bigg)^t.
$$
\end{lemma}

\begin{lemma}\label{lemma-b}
For the three polynomials $(a-q^n)(1-aq^n)$, $(b-q^n)(1-bcq^n)$, and
$(c-q^n)$ which are relatively prime to one another, there holds
\begin{align*}
&\frac{(b-q^{n})(1-bcq^{n})(c-q^{n})(x-yq^{n})}{(a-b)(a-c)(a-bc)(1-ab)(1-ac)(1-abc)}\equiv1\,\pmod{(a-q^{n})(1-aq^{n})},
\\[5pt]
&\frac{(a-q^{n})(1-aq^{n})(c-q^{n})(u-vq^{n})}{(a-b)(b-c)(a-bc)(1-ab)(1-abc)(1-bc^2)}\equiv1\pmod{(b-q^{n})(1-bcq^{n})},
\\[5pt]
&\qquad\qquad\frac{(a-q^{n})(1-aq^{n})(b-q^{n})(1-bcq^{n})}{(a-c)(b-c)(1-ac)(1-bc^2)}\equiv1\pmod{(c-q^n)},
\end{align*}
where
\begin{align*}
&x:=x(a,b,c)=a(1+a^2+a^4)(1+b^2c+bc^2)-a^2(1+a^2)(b+c+bc+b^2c^2)
\\&\qquad\qquad\qquad\:\:\,\,
-bc(1-a^3+a^6),
\\[5pt]
&y:=y(a,b,c)=a^2(1+a^2)(1+b^2c+bc^2)-a^3(b+c+bc+b^2c^2)-abc(1+a^4),
\end{align*}
\begin{align*}
&u:=u(a,b,c)=-a\{1+bc(b-c+bc+b^3c-b^2c^2+b^3c^2+b^5c^2-b^2c^3-b^4c^3)\}
\\[5pt]&\qquad\qquad\qquad\:\:\,\,
+bc(1+a^2)(1+b^2c-bc^2+b^4c^2-b^3c^3),
\\[5pt]
&v:=v(a,b,c)=-abc(1+b^2c-bc^2+b^2c^2+b^4c^2-b^3c^3)
\\[5pt]&\qquad\qquad\qquad\:\:\,
+b^2c^2(1+a^2)(1+b^2c-bc^2).
\end{align*}
\end{lemma}

Now we start to supply a  parametric generalization of Theorem
\ref{thm-a}.

\begin{theorem}\label{thm-c}
Let $n$ be a positive integer subject to $n\equiv1\pmod{3}$ and
$$B_q(k)=[6k+1]\frac{(aq,q/a,bq,q/b;q^3)_k(q;q^3)_k^2}{(q^3/a,aq^3,q^{3}/b,bq^3;q^3)_k(q^3;q^3)_k^2}q^{3k}.$$
 Then for $i,j\geq0$, modulo $\Phi_n(q)^2(1-aq^n)(a-q^n)(1-bq^n)(b-q^n)$,
\begin{align}
&\sum_{i+j\leq n-1}B_q(i)B_{q}(j)
\notag\\
  &\quad\equiv[n]^2\frac{(1-aq^n)(a-q^n)(-1-b^2+bq^n)}{(b-a)(1-ba)}\frac{(aq^2,q^2/a,q^2;q^3)_{(n-1)/3}^2}{(q^3/a,aq^3,q^3;q^3)_{(n-1)/3}^2}
\notag\\
&\qquad+[n]^2\frac{(1-bq^n)(b-q^n)(-1-a^2+aq^n)}{(a-b)(1-ab)}\frac{(bq^2,q^2/b,q^2;q^3)_{(n-1)/3}^2}{(q^3/b,bq^3,q^3;q^3)_{(n-1)/3}^2}.
\label{wei-aa}
\end{align}
\end{theorem}

\begin{proof}
Define the basic hypergeometric series (cf. \cite{Gasper}) to be
$$
_{r}\phi_{s}\left[\begin{array}{c}
a_1,a_2,\ldots,a_{r}\\
b_1,b_2,\ldots,b_{s}
\end{array};q,\, z
\right] =\sum_{k=0}^{\infty}\frac{(a_1,a_2,\ldots, a_{r};q)_k}
{(q,b_1,b_2,\ldots,b_{s};q)_k}\bigg\{(-1)^kq^{\binom{k}{2}}\bigg\}^{1+s-r}z^k.
$$
Recall the transformation formula between two $ _{8}\phi_{7}$ series
(cf. \cite[Appendix III. 23]{Gasper}):
\begin{align*}
& _{8}\phi_{7}\!\left[\begin{array}{cccccccc}
a,& qa^{\frac{1}{2}},& -qa^{\frac{1}{2}}, & b,    & c,    & d,    & e,    & f \\
  & a^{\frac{1}{2}}, & -a^{\frac{1}{2}},  & aq/b, & aq/c, & aq/d, & aq/e, &
  aq/f
\end{array};q,\, \frac{a^2q^{2}}{bcdef}
\right] \\[5pt]
&\quad =\frac{(aq,aq/ef,\lambda q/e,\lambda q/f;q)_{\infty}}
{(aq/e,aq/f,\lambda q,\lambda q/ef;q)_{\infty}}
\\[5pt]&\quad\times
{_{8}\phi_{7}}\!\left[\begin{array}{cccccccc}
\lambda,& q\lambda^{\frac{1}{2}},& -q\lambda^{\frac{1}{2}}, & \lambda b/a,    &\lambda c/a,    &\lambda d/a,    & e,    & f \\
  & \lambda^{\frac{1}{2}}, & -\lambda^{\frac{1}{2}},  & aq/b, & aq/c, & aq/d, & \lambda q/e, &
  \lambda q/f
\end{array};q,\, \frac{aq}{ef}
\right],
\end{align*}
where $\lambda=a^2q/bcd$ and $\max\{|a^2q^2/bcdef|,|aq/ef|\}<1$.
Performing the replacements $a\mapsto q^{1-n}$, $b\mapsto aq$,
$c\mapsto q/a$, $d\mapsto q/b$, $e\mapsto q/c$, $f\to bcq$,
$q\mapsto q^3$ in the last equation, we have
\begin{align}
\sum_{k=0}^{(n-1)/3}[6k+1-n]\frac{(aq,q/a,q/b,q/c,bcq,q^{1-n};q^3)_k}{(q^{3-n}/a,aq^{3-n},bq^{3-n},cq^{3-n},q^{3-n}/bc,q^3;q^3)_k}q^{(3-2n)k}
=0.\label{base-A}
\end{align}
The use of the $m=3,t=2$ case of Lemma \ref{lemma-a} leads us to
\begin{align*}
&\sum_{i+j\leq n-1}[6i+1-n]
\frac{(aq,q/a,q/b,q/c,bcq,q^{1-n};q^3)_i}{(q^{3-n}/a,aq^{3-n},bq^{3-n},cq^{3-n},q^{3-n}/bc,q^3;q^3)_i}q^{(3-2n)i}
\\[5pt]
&\qquad\times
[6j+1-n]\frac{(aq,q/a,q/b,q/c,bcq,q^{1-n};q^3)_j}{(q^{3-n}/a,aq^{3-n},bq^{3-n},cq^{3-n},q^{3-n}/bc,q^3;q^3)_j}q^{(3-2n)j}
\\[5pt]&\quad
=0.
\end{align*}
Observing that $q^n\equiv 1\pmod{\Phi_n(q)}$, it is easy to realize
that
\begin{equation}\label{wei-bb}
\sum_{i+j\leq n-1}\beta_q(i)\beta_q(j)\equiv 0\pmod{\Phi_n(q)},
\end{equation}
where
\begin{align*}
 \beta_q(k)=[6k+1]
\frac{(aq,q/a,q/b,q/c,bcq,q;q^3)_k}{(q^{3}/a,aq^{3},bq^{3},cq^{3},q^{3}/bc,q^3;q^3)_k}q^{3k}.
\end{align*}

Jackon's $_8\phi_7$ summation formula (cf. \cite[Appendix (II.
22)]{Gasper}) can be stated as
\begin{align}
& _{8}\phi_{7}\!\left[\begin{array}{cccccccc}
a,& qa^{\frac{1}{2}},& -qa^{\frac{1}{2}}, & b,    & c,    & d,    & e,    & q^{-n} \\
  & a^{\frac{1}{2}}, & -a^{\frac{1}{2}},  & aq/b, & aq/c, & aq/d, & aq/e, & aq^{n+1}
\end{array};q,q
\right] \notag\\[5pt]
&\quad
=\frac{(aq,aq/bc,aq/bd,aq/cd;q)_{n}}{(aq/b,aq/c,aq/d,aq/bcd;q)_{n}},
\label{Jackson}
\end{align}
where $a^2q=bcdeq^{-n}$. Employing the substitutions $a\mapsto q$,
$b\mapsto q/b$, $c\mapsto q/c$, $d\to bcq$, $e\mapsto q^{1+n}$,
$n\to (n-1)/3$, $q\mapsto q^3$ in \eqref{Jackson}, we find
\begin{equation}\label{base-B}
\sum_{k=0}^{(n-1)/3}\tilde{\beta}_q(k)=\frac{(q^4,q^2/b,q^2/c,bcq^2;q^3)_{(n-1)/3}}{(q,bq^3,cq^3,q^3/bc;q^3)_{(n-1)/3}},
\end{equation}
where $\tilde{\beta}_q(k)$ is the $a=q^n$ or $a=q^{-n}$ case of
$\beta_q(k)$. The utilization of the $m=3,t=2$ case of Lemma
\ref{lemma-a} yields
\begin{align*}
\sum_{i+j\leq n-1}\tilde{\beta}_q(i)\tilde{\beta}_{q}(j)
=\frac{(q^4,q^2/b,q^2/c,bcq^2;q^3)_{(n-1)/3}^2}{(q,bq^3,cq^3,q^3/bc;q^3)_{(n-1)/3}^2}.
\end{align*}
Then there is the $q$-congruence: modulo $(1-aq^n)(a-q^n)$,
\begin{align}
\sum_{i+j\leq n-1}\beta_q(i)\beta_{q}(j)
\equiv\frac{(q^4,q^2/b,q^2/c,bcq^2;q^3)_{(n-1)/3}^2}{(q,bq^3,cq^3,q^3/bc;q^3)_{(n-1)/3}^2}.
\label{wei-cc}
\end{align}

Performing the replacements $a\mapsto q$, $b\mapsto aq$, $c\mapsto
q/a$, $d\to q/c$, $e\mapsto cq^{1+n}$, $n\to (n-1)/3$, $q\mapsto
q^3$ in \eqref{Jackson}, we discover
\begin{equation}\label{base-C}
\sum_{k=0}^{(n-1)/3}\bar{\beta}_q(k)=\frac{(q^4,q^2,acq^2,cq^2/a;q^3)_{(n-1)/3}}{(cq,cq^3,aq^3,q^3/a;q^3)_{(n-1)/3}},
\end{equation}
where $\bar{\beta}_q(k)$ stands for the $b=q^n$ or $bc=q^{-n}$ case
of $\beta_q(k)$. The use of the $m=3,t=2$ case of Lemma
\ref{lemma-a} produces
\begin{align*}
\sum_{i+j\leq n-1}\bar{\beta}_q(i)\bar{\beta}_{q}(j)
=\frac{(q^4,q^2,acq^2,cq^2/a;q^3)_{(n-1)/3}^2}{(cq,cq^3,aq^3,q^3/a;q^3)_{(n-1)/3}^2}.
\end{align*}
So we obtain the $q$-congruence: modulo $(b-q^n)(1-bcq^n)$,
\begin{align}
\sum_{i+j\leq n-1}\beta_q(i)\beta_{q}(j)
\equiv\frac{(q^4,q^2,acq^2,cq^2/a;q^3)_{(n-1)/3}^2}{(cq,cq^3,aq^3,q^3/a;q^3)_{(n-1)/3}^2}.
\label{wei-dd}
\end{align}

Interchanging the parameters $b$ and $c$ in \eqref{wei-dd}, we can
deduce the $q$-congruence: modulo $(c-q^n)$,
\begin{align}
\sum_{i+j\leq n-1}\beta_q(i)\beta_{q}(j)
\equiv\frac{(q^4,q^2,abq^2,bq^2/a;q^3)_{(n-1)/3}^2}{(bq,bq^3,aq^3,q^3/a;q^3)_{(n-1)/3}^2}.
\label{wei-ee}
\end{align}

It is evident that the polynomials $(a-q^n)(1-aq^n)$,
$(b-q^n)(1-bcq^n)$, $(c-q^n)$, and $\Phi_n(q)$ are relatively prime
to one another. In view of Lemma \ref{lemma-b} and the Chinese
remainder theorem for coprime polynomials, we derive the following
result from \eqref{wei-bb}, \eqref{wei-cc}, \eqref{wei-dd}, and
\eqref{wei-ee}: modulo
$\Phi_n(q)(a-q^n)(1-aq^n)(b-q^n)(1-bcq^n)(c-q^n)$,
\begin{align}
&\sum_{i+j\leq n-1}\beta_q(i)\beta_{q}(j)
\notag\\
  &\quad\equiv\frac{(b-q^{n})(1-bcq^{n})(c-q^{n})(x-yq^{n})}{(a-b)(a-c)(a-bc)(1-ab)(1-ac)(1-abc)}
\frac{(q^4,q^2/b,q^2/c,bcq^2;q^3)_{(n-1)/3}^2}{(q,bq^3,cq^3,q^3/bc;q^3)_{(n-1)/3}^2}
\notag\\[5pt]
&\qquad+\frac{(a-q^{n})(1-aq^{n})(c-q^{n})(u-vq^{n})}{(a-b)(b-c)(a-bc)(1-ab)(1-abc)(1-bc^2)}
\frac{(q^4,q^2,acq^2,cq^2/a;q^3)_{(n-1)/3}^2}{(cq,cq^3,aq^3,q^3/a;q^3)_{(n-1)/3}^2}
\notag\\[5pt]
&\qquad+
\frac{(a-q^{n})(1-aq^{n})(b-q^{n})(1-bcq^{n})}{(a-c)(b-c)(1-ac)(1-bc^2)}
\frac{(q^4,q^2,abq^2,bq^2/a;q^3)_{(n-1)/3}^2}{(bq,bq^3,aq^3,q^3/a;q^3)_{(n-1)/3}^2}.
\label{wei-ef}
\end{align}
Noting that the factor $(1-q^n)^2$ appears in the expression
$(q^4;q^3)_{(n-1)/3}^2$, the $c=1$ case of \eqref{wei-ef} reads as
follows: modulo $\Phi_n(q)^2(a-q^n)(1-aq^n)(b-q^n)(1-bq^n)$,
\begin{align}
&\sum_{i+j\leq n-1}B_q(i)B_{q}(j)
\notag\\[5pt]
  &\quad\equiv\frac{(b-q^{n})(1-bq^{n})(1-q^{n})\{-1+a-a^2+a^3-a^4+a(1-a+a^2)q^{n}\}}{(1-a)^2(a-b)(1-ab)}
  \notag\\[5pt]
  &\qquad\times
\frac{(q^4,q^2,q^2/b,bq^2;q^3)_{(n-1)/3}^2}{(q,q^3,bq^3,q^3/b;q^3)_{(n-1)/3}^2}
\notag\\[5pt]
&\qquad+\frac{(a-q^{n})(1-aq^{n})(1-q^{n})\{-1+b-b^2+b^3-b^4+b(1-b+b^2)q^{n}\}}{(1-b)^2(b-a)(1-ba)}
 \notag\\[5pt]
  &\qquad\times
  \frac{(q^4,q^2,q^2/a,aq^2;q^3)_{(n-1)/3}^2}{(q,q^3,aq^3,q^3/a;q^3)_{(n-1)/3}^2}.
\label{wei-ff}
\end{align}
On the base of the following evaluations:
\begin{align}
&\frac{(1-q^{n})\{-1+a-a^2+a^3-a^4+a(1-a+a^2)q^{n}\}}{(1-a)^2(a-b)(1-ab)}
  \notag\\[5pt]
  &=
\frac{-(1-a)^2+(a-q^{n})(-1+a^2-a^3+aq^n-a^2q^n+a^3q^n)}{(1-a)^2(a-b)(1-ab)}
 \notag
 \end{align}
 \begin{align}
  &\equiv
\frac{-(1-a)^2-(1-a)^2a(a-q^{n})}{(1-a)^2(a-b)(1-ab)}
\notag\\[5pt]
  &=
\frac{-1-a^2+aq^n}{(a-b)(1-ab)}\pmod{(a-q^n)(1-aq^n)},
\label{base-D}
 \end{align}
 \begin{align}
&\frac{(1-q^{n})\{-1+b-b^2+b^3-b^4+b(1-b+b^2)q^{n}\}}{(1-b)^2(b-a)(1-ba)}
  \notag\\
  &\equiv
\frac{-1-b^2+bq^n}{(b-a)(1-ba)}\pmod{(b-q^n)(1-bq^n)},
\label{base-E}
\end{align}
we can rewrite \eqref{wei-ff} as \eqref{wei-aa} to complete the
proof.
\end{proof}

At present, we are ready to prove Theorem \ref{thm-a}.

\begin{proof}[Proof of Theorem \ref{thm-a}]
Selecting $b=1$ in Theorem \ref{thm-c}, we catch hold of  the
conclusion: modulo $\Phi_n(q)^4(1-aq^n)(a-q^n)$,
\begin{align}
&\sum_{i+j\leq n-1}B^{*}_q(i)B^{*}_{q}(j)
\notag\\[5pt]
  &\notag\quad\equiv
[n]^2(1-q^n)^2\frac{(q^2;q^3)_{(n-1)/3}^6}{(q^3;q^3)_{(n-1)/3}^6}
\notag\\[5pt]
  &\qquad
+[n]^2(2-q^n)
\notag\\[5pt]
&\qquad\times\bigg\{\frac{a(1-q^n)^2}{(1-a)^2}\frac{(q^2;q^3)_{(n-1)/3}^6}{(q^3;q^3)_{(n-1)/3}^6}
-\frac{(1-aq^n)(a-q^n)}{(1-a)^2}\frac{(aq^2,q^2/a,q^2;q^3)_{(n-1)/3}^2}{(q^3/a,aq^3,q^3;q^3)_{(n-1)/3}^2}\bigg\},
\label{wei-gg}
\end{align}
where $B^{*}_q(k)$ denotes the $b=1$ case of $B_q(k)$. Through the
L'H\^{o}pital rule, we get
\begin{align*}
&\lim_{a\to1}\bigg\{\frac{a(1-q^{n})^2}{(1-a)^2}\frac{(q^2;q^3)_{(n-1)/3}^4}{(q^3;q^3)_{(n-1)/3}^4}
-\frac{(1-aq^{n})(a-q^{n})}{(1-a)^2}\frac{(aq^2,q^2/a;q^3)_{(n-1)/3}^2}{(q^3/a,aq^3;q^3)_{(n-1)/3}^2}\bigg\}
\notag\\[5pt]
&\quad=\frac{(q^2;q^3)_{(n-1)/3}^4}{(q^3;q^3)_{(n-1)/3}^4}
\bigg\{q^{n}+2[n]^2\sum_{r=1}^{(n-1)/3}\bigg(\frac{q^{3rr-1}}{[3r-1]^2}-\frac{q^{3r}}{[3r]^2}\bigg)\bigg\}.
\end{align*}
Letting $a\to1$ in \eqref{wei-gg} and utilizing this limit, we prove
\eqref{eq:thm-a} in the end.
\end{proof}

\section{Proof of Theorem \ref{thm-aa}}

Firstly, we shall establish  the following  parametric
generalization of Theorem \ref{thm-aa}.

\begin{theorem}\label{thm-d}
Let $n$ be a positive integer subject to $n\equiv1\pmod{3}$. Then
for $i,j,k\geq0$, modulo
$\Phi_n(q)^2(1-aq^n)(a-q^n)(1-bq^n)(b-q^n)$,
\begin{align}
&\sum_{i+j+k\leq n-1}B_q(i)B_q(j)B_{q}(k)
\notag\\
  &\quad\equiv[n]^3\frac{(1-aq^n)(a-q^n)(-1-b^2+bq^n)}{(b-a)(1-ba)}\frac{(aq^2,q^2/a,q^2;q^3)_{(n-1)/3}^3}{(q^3/a,aq^3,q^3;q^3)_{(n-1)/3}^3}
\notag\\
&\qquad+[n]^3\frac{(1-bq^n)(b-q^n)(-1-a^2+aq^n)}{(a-b)(1-ab)}\frac{(bq^2,q^2/b,q^2;q^3)_{(n-1)/3}^3}{(q^3/b,bq^3,q^3;q^3)_{(n-1)/3}^3},
\label{wei-aaa}
\end{align}
where $B_{q}(k)$ has been given in Theorem \ref{thm-c}.
\end{theorem}

\begin{proof}
With the help of the $m=t=3$ case of Lemma \ref{lemma-a} and
\eqref{base-A}, it is routine to show that
\begin{equation}\label{wei-bbb}
\sum_{i+j+k\leq n-1}\beta_q(i)\beta_q(j)\beta_{q}(k)\equiv
0\pmod{\Phi_n(q)},
\end{equation}
where $\beta_{q}(k)$ has emerged in \eqref{wei-bb}.

In terms of the $m=t=3$ case of Lemma \ref{lemma-a} and
\eqref{base-B}, there holds
 $q$-congruence: modulo $(1-aq^n)(a-q^n)$,
\begin{align}
\sum_{i+j+k\leq n-1}\beta_q(i)\beta_q(j)\beta_{q}(k)
\equiv\frac{(q^4,q^2/b,q^2/c,bcq^2;q^3)_{(n-1)/3}^3}{(q,bq^3,cq^3,q^3/bc;q^3)_{(n-1)/3}^3}.
\label{wei-ccc}
\end{align}

Via the $m=t=3$ case of Lemma \ref{lemma-a} and \eqref{base-C},
there is
 $q$-congruence: modulo $(b-q^n)(1-bcq^n)$,
\begin{align}
\sum_{i+j+k\leq n-1}\beta_q(i)\beta_q(j)\beta_{q}(k)
\equiv\frac{(q^4,q^2,acq^2,cq^2/a;q^3)_{(n-1)/3}^3}{(cq,cq^3,aq^3,q^3/a;q^3)_{(n-1)/3}^3}.
\label{wei-ddd}
\end{align}

Interchanging the parameters $b$ and $c$ in \eqref{wei-ddd}, we can
deduce the $q$-congruence: modulo $(c-q^n)$,
\begin{align}
\sum_{i+j+k\leq n-1}\beta_q(i)\beta_q(j)\beta_{q}(k)
\equiv\frac{(q^4,q^2,abq^2,bq^2/a;q^3)_{(n-1)/3}^3}{(bq,bq^3,aq^3,q^3/a;q^3)_{(n-1)/3}^3}.
\label{wei-eee}
\end{align}

By means of Lemma \ref{lemma-b} and the Chinese remainder theorem
for coprime polynomials, we derive the following result from
\eqref{wei-bbb}-\eqref{wei-eee}: modulo
$\Phi_n(q)(a-q^n)(1-aq^n)(b-q^n)(1-bcq^n)(c-q^n)$,
\begin{align}
&\sum_{i+j+k\leq n-1}\beta_q(i)\beta_q(j)\beta_{q}(k)
\notag\\
  &\quad\equiv\frac{(b-q^{n})(1-bcq^{n})(c-q^{n})(x-yq^{n})}{(a-b)(a-c)(a-bc)(1-ab)(1-ac)(1-abc)}
\frac{(q^4,q^2/b,q^2/c,bcq^2;q^3)_{(n-1)/3}^3}{(q,bq^3,cq^3,q^3/bc;q^3)_{(n-1)/3}^3}
\notag
\end{align}
\begin{align}
&\qquad+\frac{(a-q^{n})(1-aq^{n})(c-q^{n})(u-vq^{n})}{(a-b)(b-c)(a-bc)(1-ab)(1-abc)(1-bc^2)}
\frac{(q^4,q^2,acq^2,cq^2/a;q^3)_{(n-1)/3}^3}{(cq,cq^3,aq^3,q^3/a;q^3)_{(n-1)/3}^3}
\notag\\[5pt]
&\qquad+
\frac{(a-q^{n})(1-aq^{n})(b-q^{n})(1-bcq^{n})}{(a-c)(b-c)(1-ac)(1-bc^2)}
\frac{(q^4,q^2,abq^2,bq^2/a;q^3)_{(n-1)/3}^3}{(bq,bq^3,aq^3,q^3/a;q^3)_{(n-1)/3}^3}.
\label{wei-fff}
\end{align}
The $c=1$ case of \eqref{wei-fff} can be simplified as follows:
modulo $\Phi_n(q)^2(a-q^n)(1-aq^n)(b-q^n)(1-bq^n)$,
\begin{align}
&\sum_{i+j+k\leq n-1}B_q(i)B_{q}(j)B_{q}(k)
\notag\\[5pt]
  &\quad\equiv\frac{(b-q^{n})(1-bq^{n})(1-q^{n})\{-1+a-a^2+a^3-a^4+a(1-a+a^2)q^{n}\}}{(1-a)^2(a-b)(1-ab)}
  \notag\\[5pt]
  &\qquad\times
\frac{(q^4,q^2,q^2/b,bq^2;q^3)_{(n-1)/3}^3}{(q,q^3,bq^3,q^3/b;q^3)_{(n-1)/3}^3}
\notag\\[5pt]
&\qquad+\frac{(a-q^{n})(1-aq^{n})(1-q^{n})\{-1+b-b^2+b^3-b^4+b(1-b+b^2)q^{n}\}}{(1-b)^2(b-a)(1-ba)}
 \notag\\[5pt]
  &\qquad\times
  \frac{(q^4,q^2,q^2/a,aq^2;q^3)_{(n-1)/3}^3}{(q,q^3,aq^3,q^3/a;q^3)_{(n-1)/3}^3}.
\label{wei-ggg}
\end{align}
In view of \eqref{base-D} and \eqref{base-E}, we may change
\eqref{wei-ggg} into \eqref{wei-aaa} to complete the proof.
\end{proof}

Secondly, we are going to prove Theorem \ref{thm-aa}.

\begin{proof}[Proof of Theorem \ref{thm-aa}]
Setting $b=1$ in Theorem \ref{thm-d}, we arrive at  the formula:
modulo $\Phi_n(q)^4(1-aq^n)(a-q^n)$,
\begin{align}
&\sum_{i+j+k\leq n-1}B^{*}_q(i)B^{*}_{q}(j)B^{*}_{q}(k)
\notag\\[5pt]
  &\notag\quad\equiv
[n]^3(1-q^n)^2\frac{(q^2;q^3)_{(n-1)/3}^9}{(q^3;q^3)_{(n-1)/3}^9}
\\[5pt]
  &\qquad
+[n]^3(2-q^n)
\notag\\[5pt]
&\qquad\times\bigg\{\frac{a(1-q^n)^2}{(1-a)^2}\frac{(q^2;q^3)_{(n-1)/3}^9}{(q^3;q^3)_{(n-1)/3}^9}
-\frac{(1-aq^n)(a-q^n)}{(1-a)^2}\frac{(aq^2,q^2/a,q^2;q^3)_{(n-1)/3}^3}{(q^3/a,aq^3,q^3;q^3)_{(n-1)/3}^3}\bigg\},
\label{wei-hhh}
\end{align}
where $B^{*}_q(k)$ is the $b=1$ case of $B_q(k)$. Through the
L'H\^{o}pital rule, we find
\begin{align*}
&\lim_{a\to1}\bigg\{\frac{a(1-q^{n})^2}{(1-a)^2}\frac{(q^2;q^3)_{(n-1)/3}^6}{(q^3;q^3)_{(n-1)/3}^6}
-\frac{(1-aq^{n})(a-q^{n})}{(1-a)^2}\frac{(aq^2,q^2/a;q^3)_{(n-1)/3}^3}{(q^3/a,aq^3;q^3)_{(n-1)/3}^3}\bigg\}
\end{align*}
\begin{align*}
&\quad=\frac{(q^2;q^3)_{(n-1)/3}^6}{(q^3;q^3)_{(n-1)/3}^6}
\bigg\{q^{n}+3[n]^2\sum_{r=1}^{(n-1)/3}\bigg(\frac{q^{3r-1}}{[3r-1]^2}-\frac{q^{3r}}{[3r]^2}\bigg)\bigg\}.
\end{align*}
Letting $a\to1$ in \eqref{wei-hhh} and drawing support from the
above limit, we obtain the consequence: modulo $\Phi_n(q)^6$,
\begin{align*}
\sum_{i+j+k\leq n-1}A_q(i)A_q(j)A_{q}(k)&\equiv
[n]^3\frac{(q^2;q^3)_{(n-1)/3}^9}{(q^3;q^3)_{(n-1)/3}^9}
\notag\\
&\quad\times\:
\bigg\{1+3[n]^2(2-q^n)\sum_{r=1}^{(n-1)/3}\bigg(\frac{q^{3r-1}}{[3r-1]^2}-\frac{q^{3r}}{[3r]^2}\bigg)\bigg\}.
\end{align*}
Noticing the fact $[n]^5(2-q^n)\equiv [n]^5\pmod{\Phi_n(q)^6}$, we
are led to \eqref{eq:thm-aa} finally.
\end{proof}

\textbf{Acknowledgments}\\

 The work is supported by Hainan Provincial Natural Science Foundation of China (No. 124RC511) and the National Natural Science Foundation of China (No. 12071103).


\begin{thebibliography}{10}
\small \setlength{\itemsep}{-.8mm}


\bibitem{Mohamed} M. El Bachraoui, On supercongruences for truncated sums of squares of basic hypergeometric series,
Ramanujan J. 54 (2021), 415--426.

\bibitem{Mohamed-a} M. El Bachraoui, $N$-tuple sum analogues for Ramanujan-type congruences.
Proc. Am. Math. Soc. 151 (2023), 1--16.

\bibitem{Gasper}G. Gasper, M. Rahman, Basic Hypergeometric Series (2nd edition),
Cambridge University Press, Cambridge, 2004.

\bibitem{Guo-a} V.J.W. Guo, Some $q$-analogues of supercongruences for truncated $_3F_2$ hypergeometric series, Ramanujan J. 59 (2022), 131--142.


\bibitem{Guo-b} V.J.W. Guo, Some $q$-supercongruences from the Gasper and Rahman quadratic summation, Rev. Mat. Complut. 36 (2023), 993--1002.

\bibitem{Guo-c} V.J.W. Guo, Some $q$-supercongruences from the Gasper and Rahman quadratic summation, Rev. Mat. Complut. 36 (2023), 993--1002.

\bibitem{GuoLi} V.J.W. Guo, L. Li, $q$-Supercongruences from squares of basic hypergeometric series,
Rev. R. Acad. Cienc. Exactas F\'is. Nat., Ser. A Mat. 117 (2023),
Art. 26.

\bibitem{GLS} V.J.W. Guo, J.-C. Liu, M.J. Schlosser, An extension of a supercongruence of Long and Ramakrishna,
 Proc. Amer. Math. Soc. 151 (2023), 1157--1166.


\bibitem{GS}V.J.W. Guo, M.J. Schlosser, Some $q$-supercongruences from transformation formulas for basic
hypergeometric series, Constr. Approx. 53 (2021), 155--200.

\bibitem{GS-a}V.J.W. Guo, M.J. Schlosser, Three families of q-supercongruences modulo the square and cube of a cyclotomic polynomial,
Rev. R. Acad. Cienc. Exactas F\'is. Nat., Ser. A Mat. 117 (2023),
Art. 9.


\bibitem{GuoZu-b}V.J.W. Guo, W. Zudilin, A $q$-microscope for supercongruences, Adv. Math. 346 (2019), 329--358.

\bibitem{GuoZu-c}V.J.W. Guo, W. Zudilin, Dwork-type supercongruences through a creative $q$-microscope,
 J. Combin. Theory, Ser. A 178 (2021), Art. 105362.

\bibitem{Li}L. Li, Some $q$-supercongruences for truncated forms of squares of basic hypergeometric series,
J. Difference Equ. Appl. 27 (2021), 16--25.

\bibitem{LP}J.-C. Liu, F. Petrov, Congruences on sums of $q$-binomial coefficients, Adv. Appl. Math. 116 (2020), Art.~102003.

\bibitem{LR} L. Long, R. Ramakrishna, Some supercongruences occurring in truncated hypergeometric series, Adv. Math. 290 (2016), 773--808.

\bibitem{NW}H.-X. Ni, L.-Y. Wang, Two $q$-supercongruences from Watson's transformation,
Rev. R. Acad. Cienc. Exactas F\'is. Nat., Ser. A Mat. 116 (2022),
Art. 30.

\bibitem{Hamme} L. Van Hamme, Some conjectures concerning partial sums of
generalized hypergeometric series, in: p-Adic Functional Analysis
(Nijmegen, 1996), Lecture Notes in Pure and Appl. Math. 192, Dekker,
New York, 1997, pp. 223--236.

\bibitem{WY-a}X. Wang, C. Xu, $q$-Supercongruences on triple and
quadruple sums, Results Math. 78 (2023), Art.~27.

\bibitem{Wei-a} C. Wei, $q$-Supercongruences from Jackson's $_8\phi_7$ summation and Watson's $_8\phi_7$ transformation,
J. Combin. Theory, Ser. A 204 (2024), Art. 105853.

\bibitem{Wei-b} C. Wei, C. Li, Some $q$-supercongruences for double and triple basic hypergeometric series,
 Proc. Amer. Math. Soc. 152 (2024), 2283--2296.


\end{thebibliography}
\end{document}